\documentclass{au}

 \presentedby{\dots}
 \received{\dots}{\dots}

\usepackage{graphicx}
\usepackage{newlattice}

\DeclareMathOperator{\Princ}{Princ}

\DeclareMathOperator{\Base}{Base}

\DeclareMathOperator{\Conc}{\tup{Con}_\tup{c}}
\newcommand{\Pd}{P^{\,\tup{d}}}

\theoremstyle{plain}
\newtheorem{theorem}{Theorem}

\newtheorem{lemma}[theorem]{Lemma}

\theoremstyle{definition}

\newtheorem{problem}{Problem}

\begin{document}
\title{The order of principal congruences of a bounded lattice}  
\author{G. Gr\"{a}tzer} 
\email[G. Gr\"atzer]{gratzer@me.com}
\urladdr[G. Gr\"atzer]{http://server.maths.umanitoba.ca/homepages/gratzer/}
\address{Department of Mathematics\\
  University of Manitoba\\
  Winnipeg, MB R3T 2N2\\
  Canada}


\begin{abstract}
We characterize the order of principal congruences 
of a bounded lattice as a bounded ordered set.
We~also state a number of open problems in this new field.
\end{abstract}

\maketitle

\section{Introduction}\label{S:Intro}

\subsection{Congruence lattices}\label{S:conglattices}
Let $A$ be a lattice (resp., join-semilattice with zero). 
We call $A$ \emph{representable} 
if there exist a lattice $L$
such that $A$ is isomorphic to the congruence lattice of $L$, 
in formula, $A \iso \Con L$ 
(resp., $A$ is isomorphic to the join-semilattice with zero
of compact congruences of $L$, 
in~formula, $A \iso \Conc L$).

For over 60 years, one of lattice theory's most central conjectures 
was the following:
%
\begin{quote}
\emph{Characterize representable lattices as distributive algebraic lattices.}
\end{quote}
(Or equivalently:
Characterize representable join-semilattices as distributive join-semilattice with zero.)
This conjecture was refuted in F. Wehrung \cite{CLP}.

The finite case of this field 
is surveyed in my book \cite{CFL}.
The infinite case---along with some research fields 
connected with it---is surveyed 
in four chapters in \cite{LTSTA}, 
three by F.~Wehrung and one by me.

\subsection{Principal congruences}\label{S:principal}
In this note, we deal with $\Princ L$, 
the order of principal congruences of a lattice $L$.
Observe that 
\begin{enumeratea}

\item $\Princ L$ is a directed order with zero.

\item $\Conc L$ is the set of compact elements of $\Con L$, 
a lattice theoretic characterization of this subset.

\item $\Princ L$ is a directed subset of $\Conc L$ 
containing the zero and join-generat\-ing $\Conc L$;
there is no lattice theoretic characterization of this subset.
\end{enumeratea}
Figure~\ref{F:N7} shows the lattice $\SN 7$ 
and its congruence lattice $\SB 2 + 1$.
Note that $\Princ \SN 7 = {\Con \SN 7} - \set{\bgg}$.
While in the standard representation~$K$
of $\SB 2 + 1$ as a congruence lattice
(G. Gr\"atzer and E.\,T. Schmidt \cite{GS62}; 
see also in my books \cite{CFL}, \cite{LTF}), 
we have $\Princ K = \Con K$. 
This example shows that $\Princ L$ has no lattice theoretic description in $\Con L$.

\begin{figure}[htb]
\centerline{\includegraphics[scale=1.0]{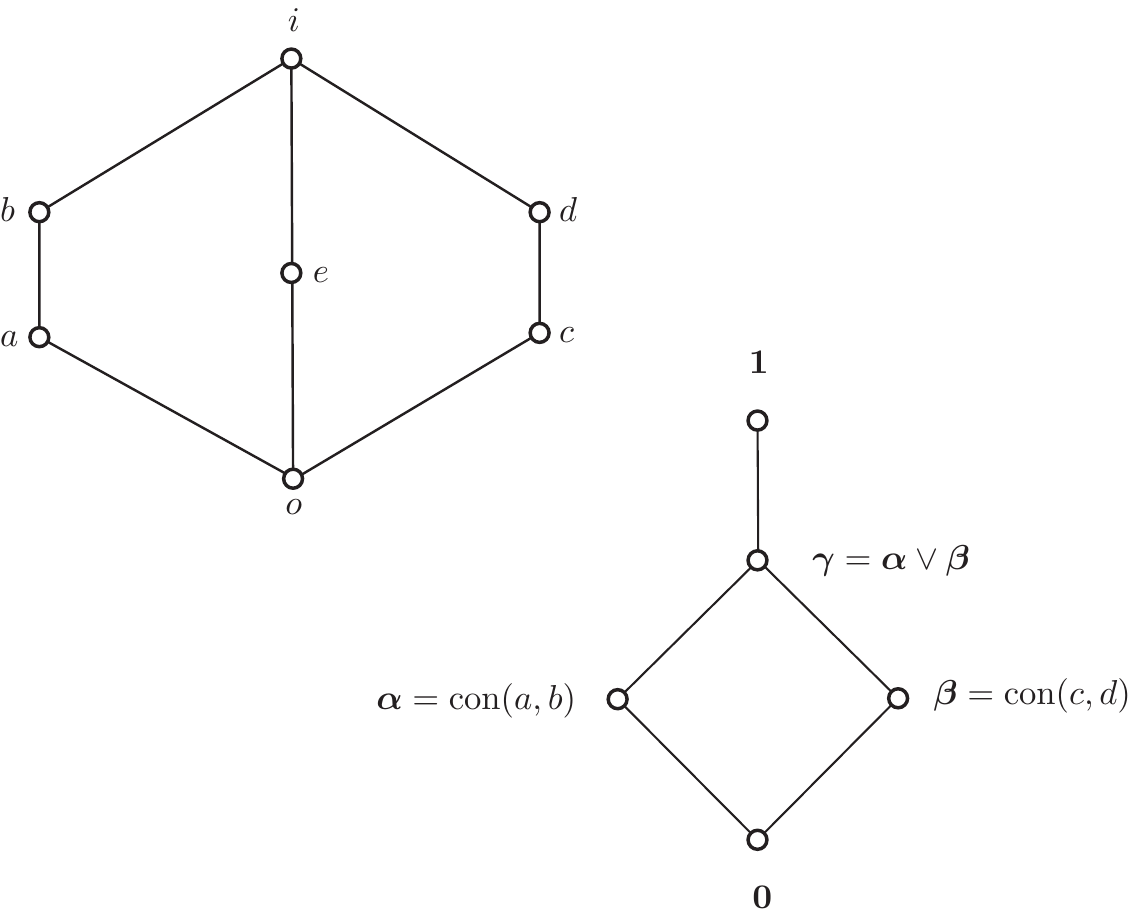}}
\caption{The lattice $\SN 7$ and its congruence lattice}\label{F:N7}
\end{figure}

It was pointed out in G. Gr\"atzer and E.\,T. Schmidt \cite{GS60a} 
that for every algebra $\F A$ we can construct
an algebra $\F B$
such that $\Con \F A \iso \Con \F B$  
and $\Princ \F B = \Conc {\F B}$.

For a long time, we have tried 
to prove such a result for lattices
but we have been unable to construct even
a \emph{proper} congruence-preserving extension
for a general lattice; see the discussion
in G. Gr\"atzer and E.\,T. Schmidt \cite{GS95d}.
This logjam was broken 
in G. Gr\"atzer and F. Wehrung \cite{GW99}
by introducing the boolean triple construction.
G. Gr\"atzer and E.\,T. Schmidt \cite{GS01b} 
uses this construction to prove the following result:

\begin{theorem}\label{T:old}
Every lattice $L$ has a 
congruence-preserving extension $K$
satisfying 
\[
   \Princ K = \Conc K.
\]
\end{theorem}

So if a distributive join-semilattice with zero $S$
can be represented as $\Conc L$ for a lattice $L$, then
$S$ can also be represented as $\Princ K$ for a lattice $K$.
This is a further illustration of the fact that
$\Princ L$ has no lattice theoretic description
in $\Con L$.

\subsection{The result}\label{S:result}

For a bounded lattice $L$, the order $\Princ K$ is bounded. 
We now state the converse.

\begin{theorem}\label{T:bounded}
Let $P$ be an order with zero and unit.
Then there is a bounded lattice~$K$ such that
\[
   P \iso \Princ K.
\]
If $P$ is finite, we can construct $K$ as a finite lattice.
\end{theorem}

We construct $K$ as a lattice of length $5$.
So $K$ is complete.
All of its congruences are complete.
So we also get Theorem~\ref{T:bounded}
for principal congruences of \emph{complete lattices}
and for principal \emph{complete congruences} of complete lattices.

\subsection{Problems}
The real purpose of this note is to state some of the many
open problems in this field.

\subsubsection{General lattices}

\begin{problem}\label{P:Prince}
Can we characterize the order $\Princ L$ for a lattice $L$
as a directed order with zero?
\end{problem}

Even more interesting would be to charaterize 
the pair $P = \Princ L$ in $S =\Conc L$ by the properties
that $P$ is a directed order with zero 
that join-generates~$S$. 
We have to rephrase this 
so it does not require a solution 
of the congruence lattice characterization problem.

\begin{problem}\label{P:Princ in Conc}
Let $S$ be a representable join-semilattice. 
Let $P \ci S$ be a directed order with zero 
and let $P$ join-generate~$S$.
Under what conditions is there a lattice $K$ such 
that $\Conc K$ is isomorphic to $S$
and under this isomorphism $\Princ K$ corresponds to $P$?
\end{problem}

For a lattice $L$, let us define a \emph{valuation} $v$ 
on $\Conc L$ as follows: 
for a compact congruence \bga of $L$, 
let $v(\bga)$ be the smallest integer $n$
such that the congruence \bga is the join of $n$ principal congruences.
A valuation $v$ has some obvious properties, for instance,
$v(\zero) = 0$ and $v(\bga \jj \bgb) \leq v(\bga) + v(\bgb)$.
Note the connection with $\Princ L$:
\[
   \Princ L = \setm{\bga \in \Conc L}{v(\bga) \leq 1}.
\]

\begin{problem}\label{P:valuation}
Let $S$ be a representable join-semilattice. 
Let $v$ map $S$ to the natural numbers.
Under what conditions is there an isomorphism \gf
of $S$ with $\Conc K$ for some lattice $K$
so that under \gf the map $v$ corresponds 
to the valuation on $\Conc K$?
\end{problem}

\subsubsection{Finite lattices}
Let $D$ be a finite distributive lattice.
In G. Gr\"atzer and E.\,T. Schmidt \cite{GS62},
we represent $D$ as the congruence lattice
of a finite lattice~$K$ in which 
\emph{all congruences are principal} (that is, $\Con K = \Princ K$).


\begin{problem}\label{P: Q in Con}
Let $D$ be a finite distributive lattice. 
Let $Q$ be a subset of $D$ satisfying
$\set{0, 1} \uu \Ji K \ci Q \ci D$.
When is there a finite lattice~$K$ such 
that $\Con K$ is isomorphic to $D$
and under this isomorphism $\Princ K$ 
corresponds to $Q$?
\end{problem}

In the finite variant of Problem~\ref{P:valuation},
we need an additional property.

\begin{problem}\label{P:finite valuation}
Let $S$ be a finite distributive lattice. 
Let $v$ be a map of $D$ to the natural numbers
satisfying $v(0) = 0$, $v(1) = 1$, and 
$v(a \jj b) \leq v(a) + v(b)$ for $a, b \in D$.
When is there an isomorphism \gf
of $D$ with $\Con K$ for some finite lattice~$K$
such that under \gf the map $v$ corresponds 
to the valuation on $\Con K$?
\end{problem}

\subsubsection{Special classes of lattices}

There are many problems 
that deal with lattices with
we only mention two.

In G. Gr\"atzer, H. Lakser and E.\,T. Schmidt \cite{GLS98a} 
and G. Gr\"atzer and E.\,T. Schmidt \cite{GS01a}, 
we investigate congruence lattices of finite semimodular lattices.

\begin{problem}\label{P:semimod}
In Theorem~\ref{T:bounded}, can we construct a
semimodular lattice?
\end{problem}

\begin{problem}\label{P:semimodproblems}
In Problems \ref{P:Princ in Conc} and \ref{P:valuation}, 
in the finite case, 
can we construct a finite semimodular lattice $K$?
\end{problem}

In E.~T. Schmidt \cite{tS62}
(see also G. Gr\"atzer and E.\,T. Schmidt \cite{GS03a}),
for a finite distributive lattice $D$, 
a countable modular lattice $M$ is constructed
with $\Con M \iso D$.

\begin{problem}\label{P:mod}
In Theorem~\ref{T:bounded}, for a finite $P$,
can we construct a countable modular lattice $K$?
\end{problem}

For the background for some other classes of lattices,
see my book \cite{CFL}.

\subsubsection{Complete lattices}

The techniques developed in this note may be applicable
to solve the following problem:

\begin{problem}\label{P:bounded}
Let~$K$ be a bounded lattice. Does there exist a complete lattice~$L$ such that $\Con{K}\cong\Con{L}$?
\end{problem}

\subsubsection{Algebras in general}\label{S:Algebras}

Some of these problems seem to be of interest for algebras
other than lattices as well.

\begin{problem}\label{P:PrinceAlg}
Can we characterize the order $\Princ \F A$ 
for an algebra~$\F A$
as an order with zero?
\end{problem}

\begin{problem}\label{P:unit}
For an algebra $\F A$,
how is the assumption that the unit congruence $\one$
is compact reflected in  the order $\Princ \F A$? 
\end{problem}

\begin{problem}\label{P:PrinceAlgStrong}
Let  $\F A$ be an algebra
and let $\Princ \F A \ci Q \ci \Conc \F A$. 
Does there exist an algebra $\F B$ such that
$\Con \F A \iso \Con \F B$ and under this isomorphism
$Q$ corresponds to $\Princ \F B$? 
\end{problem}

\begin{problem}\label{P:ValAlg}
Extend the concept of valuation to algebras.
State and solve Problem~\ref{P:valuation} 
for algebras.
\end{problem}

\begin{problem}\label{P:congpres}
Can we sharpen the result of 
G. Gr\"atzer and E.\,T. Schmidt \cite{GS60a}:
every algebra $\F A$ 
has a congruence-preserving extension $\F B$
such that $\Con \F A \iso \Con \F B$  
and $\Princ \F B = \Conc {\F B}$.
\end{problem}

I do not even know whether every algebra $\F A$ has a 
proper congruence-preserving extension $\F B$.

\section{The construction}\label{S:construction}
For a bounded order $Q$, 
let $Q^-$ denote the order $Q$ with the bounds removed.
Let $P$ be the order in Theorem~\ref{T:bounded}. 
Let $0$ and $1$ denote the zero and unit of $P$, respectively.
We denote by $\Pd$ those elements of $P^-$ that are not comparable to any other element of $P^-$, that is,
\[
   \Pd = \setm{x \in P^-}{ x \parallel y 
      \text{ for all } y \in P^-,\ y \neq x}.
\]

\subsection{The lattice $F$}\label{S:F}
We first construct the lattice $F$
consisting of the elements \text{$o$, $i$} 
and the elements $a_p, b_p$ for every $p \in P$,
where $a_p \neq b_p$ for every $p \in P^-$
and $a_0 = b_0$, $a_1 = b_1$. These elements are ordered and
the lattice operations are formed as in Figure~\ref{F:F}. 
\begin{figure}[h!]
\centerline{\includegraphics[scale=1.0]{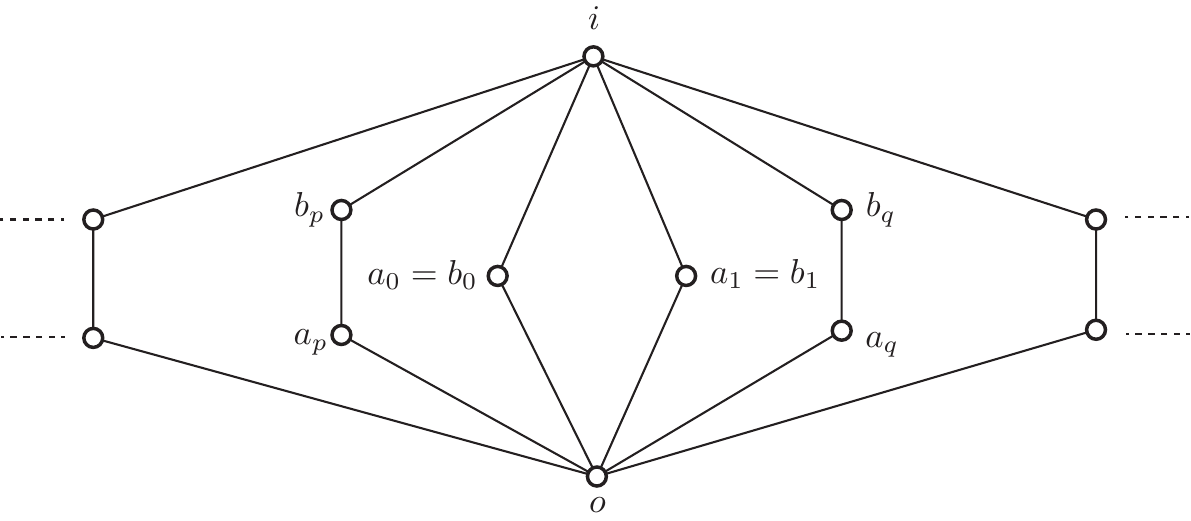}}
\caption{The lattice $F$}\label{F:F}
\end{figure}

\subsection{The lattice $K$}\label{S:K}
We are going to construct the lattice $K$ 
(of Theorem~\ref{T:bounded})
as an extension of $F$. 
The principal congruence of $K$ representing $p \in P^-$ 
will be $\con{a_p, b_p}$.
\begin{figure}[hbt]
\centerline{\includegraphics[scale=1.0]{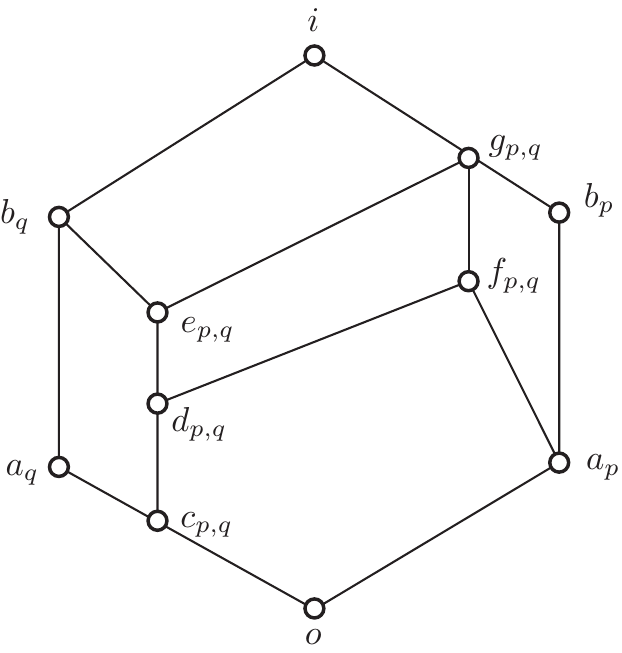}}
\caption{The lattice $S = S(p, q)$}\label{F:S}
\end{figure}

\begin{enumeratea}

\item We add the set
\[
   \set{c_{p,q},d_{p,q},e_{p,q},f_{p,q},g_{p,q}}
\]
to the sublattice 
\[
   \set{o, a_p, b_p, a_q, b_q, i}
\]
of $F$ for $p < q \in P^-$ to form the sublattice $S(p,q)$, 
as illustrated in Figure~\ref{F:S}. 

\item For $p \in \Pd$, let $C_p = \set{o, a_p, b_p, i}$,
a four-element chain.

\item We define the set
\begin{equation*}
   K = \UUm{S(p, q)}{p < q \in P^-} \uu 
      \UUm{C_p}{p \in \Pd}\uu \set{a_0, a_1}.
\end{equation*}
\end{enumeratea}
Now we are ready to define the lattice $K$.

\begin{figure}[p]
\centerline{\includegraphics[scale=.85]{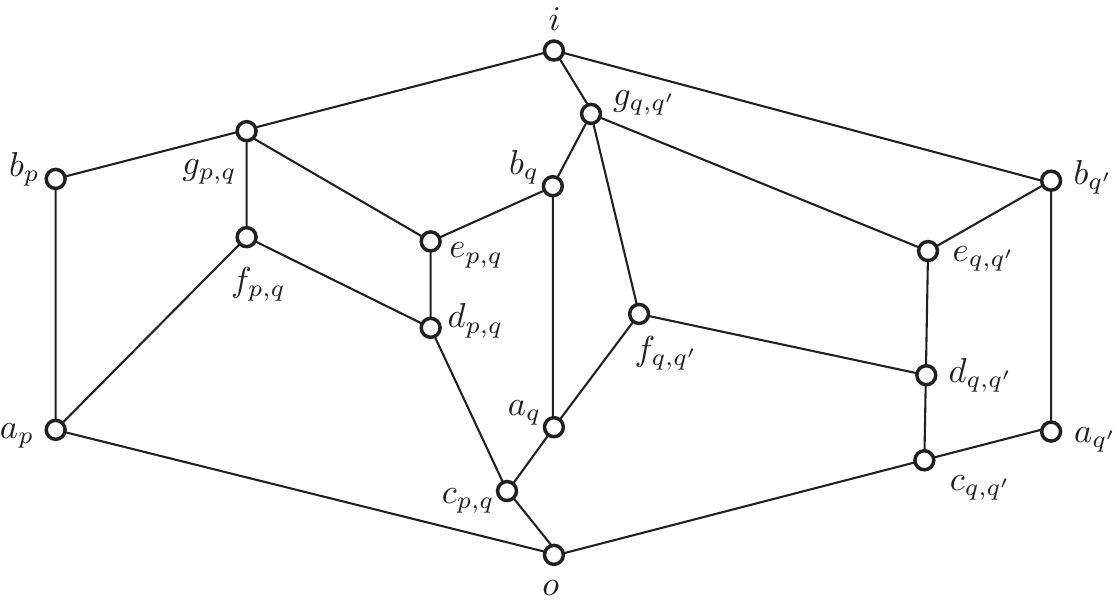}}
\caption{The lattice $S_{\SC{}} = S(p < q,\ q < q')$}\label{F:C}
\bigskip

\bigskip

\bigskip

\centerline{\includegraphics[scale=.85]{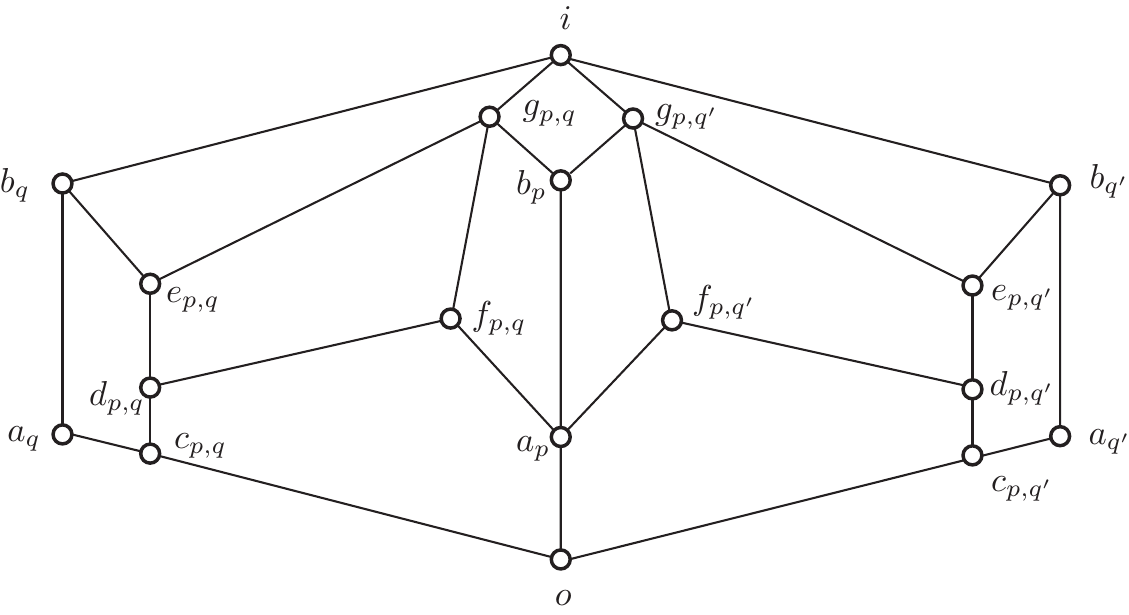}}
\caption{The lattice $S_{\SV{}} = S(p < q,\ p < q')$
with $q \neq q'$}\label{F:V}
\bigskip

\bigskip

\bigskip

\centerline{\includegraphics[scale=.80]{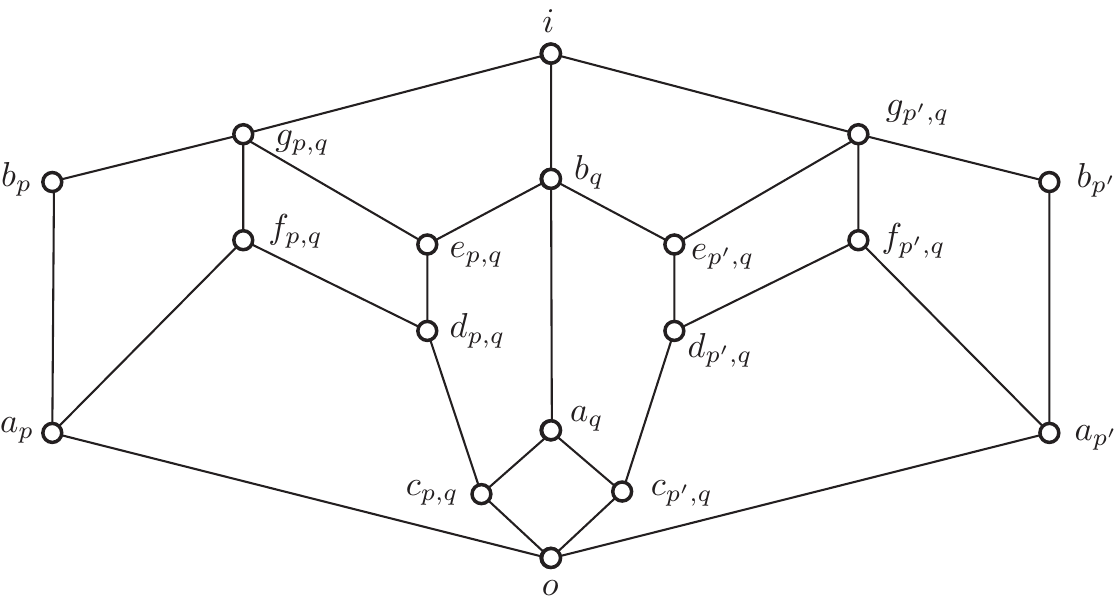}}
\caption{The lattice $S_{\SH{}} = S(p < q,\ p' < q)$
with $p \neq p'$}\label{F:H}
\end{figure}

We make the set $K$ into a lattice by 
defining the operations $\jj$ and $\mm$ with the
following nine rules. 
\begin{enumeratei}

\item The operations $\jj$ and $\mm$
are idempotent and commutative and $o$ is the zero and $i$ is the unit of $K$.

\item For $p \in \Pd$ and $x, y \in C_p \ci K$,  
we define $x \jj y$, $x \mm y$ in $K$ 
as in the chain~$C_p$. (So $C_p$~is a sublattice of $K$.)

\item  For $p < q \in P^-$ and $x, y \in S(p, q) \ci K$,
we define $x \jj y$, $x \mm y$ in $K$ as in the lattice
$S(p, q)$.  (So $S(p, q)$~is a sublattice of $K$.)

\item For $p \in \Pd$, $x \in C_p^-$, and $y \in K - C_p$,
the elements $x$ and $y$ are com\-ple\-mentary in~$K$, that is,
$x \jj y = i$ and $x \mm y = o$. 

\item  For $x = a_0$ and for $x = a_1$,
the element $x$ is complementary to any element 
$y \neq x \in K^-$.

In the following four rules, 
let $p < q,\ p' < q' \in P^-$,
$x \in S(p, q)^-$, and $y \in S(p', q')^-$.
By rule (iii), we can assume that
$\set{p, q} \neq \set{p', q'}$.

\item If $\set{p, q} \ii \set{p', q'} = \es$, 
then the elements $x$ and $y$ are complementary in~$K$.

\item If $q = p'$, 
we form $x \jj y$ and $x \mm y$ in $K$
in the lattice 
\[
   S_{\SC{}} = S(p < q,\ q < q'),
\] 
illustrated in Figure \ref{F:C}.

\item  If $p = p'$ and $q \neq q'$, 
we form $x \jj y$ and $x \mm y$ in $K$
in the lattice 
\[
   S_{\SV{}} = S(p < q,\ p < q'),
\]  
illustrated in Figure \ref{F:V}.

\item If $q = q'$ and $p \neq p'$, 
we form $x \jj y$ and $x \mm y$ in $K$
in the lattice 
\[
   S_{\SH{}} = S(p < q,\ p' < q),
\] 
illustrated in Figure \ref{F:H}. 
\end{enumeratei}

In the last three rules, $\SC{}$ for chain,
$\SV{}$ for V-shaped, 
$\SH{}$ for Hat-shaped
refer to the shape of the three element order
$\set{p, q} \uu \set{p', q'}$ in $P^-$.

Observe that Rules (vi)--(ix) exhaust all possibilities
under the assumption $\set{p, q} \neq \set{p', q'}$.

Note that
\begin{align*}
           S &= S(p, q),\\
   S_{\SC{}} &= S(p < q,\ q < q'),\\
   S_{\SV{}} &= S(p < q,\ p < q'),\\ 
   S_{\SH{}} &= S(p < q,\ p' < q) 
\end{align*}
are sublattices of $K$. 

Informally, these rules state that to form $K$, 
we add elements to $F$ so that we get the sublattices 
listed in the previous paragraph.

Alternatively, we could have defined the ordering on $K$.
Note that the ordering is larger than
\begin{equation*}
   \UUm{\leq_{S(p, q)}}{p < q \in P^-} \uu 
      \UUm{\leq_{C_p}}{p \in \Pd}.
\end{equation*}

\section{The proof}\label{S:proof}
\subsection{Preliminaries}\label{S:Preliminaries}
It is easy, if somewhat tedious, 
to verify that $K$ is a lattice. 
Note that all our sublattice constructs are bounded planar orders, 
hence planar lattices.
We have to describe the congruence structure of $K$.

Let $L$ be a lattice with $0$ and $1$. 
A congruence block of $L$ is \emph{trivial} if it is a singleton.

A \emph{$\set{0, 1}$-isolating congruence} 
$\bga$ of~$L$ (an~\emph{I-congruence}, for short), 
is a congruence $\bga > \zero$, 
such that $\set{0}$ and $\set{1}$ are 
(trivial) congruence blocks of~\bga. 

If $|P| \leq 2$, then we can construct $K$ 
as a one- or two-element chain. So for the proof,
we assume that  $|P| > 2$, that is, $P^- \neq \es$.

\begin{lemma}\label{L:big}
For every $x \in K^-$, there is an 
$\set{o,i}$-sublattice $A$ of $K$ containing~$x$ and 
isomorphic to $\SM 3$.
\end{lemma}

\begin{proof}
Since $P^- \neq \es$ by assumption, we can choose $p \in P^-$. 
If $x \in \set{a_0, a_1}$, 
then \[A = \set{a_p, a_0, a_1, o, i}\] is such a sublattice.
If $x \nin \set{a_0, a_1}$, 
then \[A = \set{x, a_0, a_1, o, i}\] is such a sublattice.
\end{proof}

\begin{lemma}\label{L:nonisolating}
Let us assume that \bga is not an I-congruence of $K$.
Then $\bga = \one$.
\end{lemma}

\begin{proof}
Indeed, if \bga is not an I-congruence of $K$, 
then there is an $x \in K^-$ such that $\cng x = o (\bga)$
or $\cng x = o (\bga)$. Using the sublattice $A$
provided by Lemma~\ref{L:big}, 
we conclude that $\bga = \one$,
since $A$ is a simple $\set{o,i}$-sublattice.
\end{proof}

\subsection{The congruences of $S$}\label{S:Scongs}
We start with the congruences of the lattice $S = S(p, q)$ 
with $p < q \in P^-$, see Figure~\ref{F:S}.

\begin{lemma}\label{L:Scongs}
The lattice $S = S(p, q)$ has two I-congruences:
\[
   \con{a_p, b_p} < \con{a_q, b_q},
\]
see Figure~\ref{F:Scongs}.
\end{lemma}

\begin{figure}[bth]
\centerline{\includegraphics[scale=0.9]{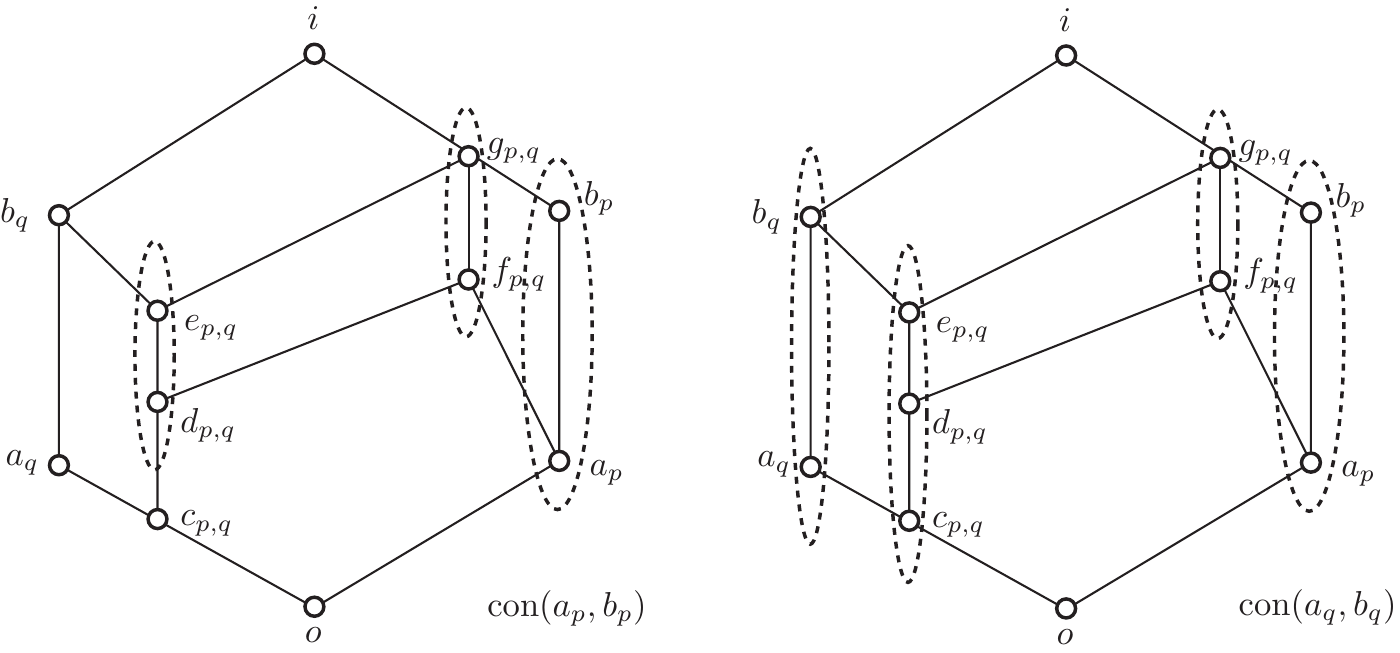}}
\caption{The I-congruences of $S = S(p, q)$ }\label{F:Scongs}
\end{figure}

\begin{proof}
An easy computation. 
First, check that Figure \ref{F:Scongs} correctly
describes the two join-irreducible I-con\-gruences
$\con{a_p, b_p}$ and $\con{a_q, b_q}$.
Then, check all 12 prime intervals $[x, y]$
and show that $\con{x, y}$ is either not an I-con\-gruence 
or equals $\con{a_p, b_p}$ or $\con{a_q, b_q}$.
For instance, $\con{d_{p,q}, e_{p,q}} = \con{a_p, b_p}$
and $[b_p, g_{p,q}]$ is not an I-con\-gruence because  
$\cng c_{p,q}=o(\con{b_p, g_{p,q}})$.
The other 10 cases are similar. 
Finally, note that the
two join-irreducible I-congruences we found are comparable, 
so there are no join-reducible I-congruences.
\end{proof}

Clearly, $S(p, q)/\con{a_q, b_q} \iso \SC 2 \times \SC 3$.

\subsection{The congruences of $K$}\label{S:Kcongs}

For $p \in \Pd$, let $\bge_p$ denote the congruence
$\con{a_p, b_p}$ on $K$.

Let $H \ci \Pd$. Let $\bge_H$ denote the equivalence relation
\[
    \bge_H = \JJm{\bge_p}{p \in H}
\] 
on $K$.

Let \bgb be an I-congruence of the lattice $K$. 
We associate with \bgb a subset of the order $P^-$:
\[
   \Base(\bgb) = \setm{p \in P^-}{\cngd a_p=b_p(\bgb)}.
\]

\begin{lemma}\label{L:downset}
Let \bgb be an I-congruence of the lattice $K$.
Then $\Base(\bgb)$ is a down set of $P^-$.
\end{lemma}

\begin{proof}
Let $p < q \in P$ and let $q \in \Base(\bgb)$.
Then $\cng a_q=b_q(\bgb)$. 
By Lemma~\ref{L:Scongs} (see also Figure~\ref{F:Scongs}),
$\cng a_p=b_p(\bgb)$, so $p \in \Base(\bgb)$, 
verifying that $\Base(\bgb)$ is a down set.
\end{proof}

Let $H$ be a down set of $P^-$. 
We define the binary relation:
\[
   \bgb_H = \bge_H \uu \UUm{\consub{S(p,q)}{a_q,b_q}}{q \in H} 
            \uu \UUm{\consub{S(p,q)}{a_p,b_p}}{p \in H}.
\]
\begin{lemma}\label{L:cong}
$\bgb_H$ is an I-congruence on $K$.
\end{lemma}

Note that $\bgb_\es = \zero$.

\begin{proof}
$\bgb_H$ is reflexive and symmetric.
It clearly leaves $o$ and $i$ isolated. 

It is easy to verify that 
$\bgb_H$ classes are pairwise disjoint 
two- and three-element chains, so $\bgb_H$ is transitive
and hence an equivalence relation.

We verify the Substitution Properties. 
By Lemma I.3.11 of \cite{LTF}, it is sufficient to verify
that if $x < y \in K$,  
and $\cng x = y (\bgb_H)$,
then $\cng x \jj z = y \jj z(\bgb_H)$ and 
$\cng x \mm z = y \mm z(\bgb_H)$ for $z \in K$.

So let $x < y \in K^-$ and $\cng x = y (\bgb_H)$.
Then $x < y \in S(p, q)^-$, for some  $p < q \in P^-$, and 
\[\cng x = y (\consub{S(p, q)}{a_q, b_q})\] with $q \in H$,
or \[\cng x = y (\consub{S(p, q)}{a_p, b_p})\] with $p \in H$.

Let $z \in S(p', q')^-$ with $p' < q' \in P^-$.

If $\set{p, q} = \set{p', q'}$, 
the Substitution Properties for $\bgb_H$ in $K$ follow 
from the Substitution Properties for $\con{a_p, b_p}$  in $S(p, q)$.

If $\set{p, q} \ii \set{p', q'} = \es$, 
then by Rule (vi), 
the elements $x$, $z$, and $y$, $z$
are complementary, so the Substitution Properties are trivial.

Otherwise, $\set{p, q} \uu \set{p', q'}$
has three elements. So we have three cases to consider.

Case $\SC{}$: $p < p' = q < q'$ (or symmetrically, $p' < q' = p < q$).

Case $\SV{}$: $p = p' < q$, $p = p' < q'$, $q \neq q'$.

Case $\SH{}$: $p < q = q'$, $p' < q = q'$, $p \neq p'$.

To verify Case $\SC{}$, utilize Figure~\ref{F:C}.
Since 
\[
   x \leq y \in S(p, q)^- \ci S(p < q,\ q < q')^-
\]
and 
\[
   z \in S(q, q')^- \ci S(p < q,\ q < q')^-,
\]
there is only way (SP${}_\jj$) can fail: $x \jj z < y \jj z$.

We can assume that $z \nin S(p < q)$, 
so $x\jj z, y \jj z \nin S(p<q)$. 
If $q \in H$, 
then there is only one case to check for the I-congruence $\bgb_H$: 
\[
   (x \jj z,y \jj z)=(f_{q, q'}, g_{q, q'}) 
           \in \con{a_q, b_q} \ci \bgb_H.
\] 
If $q \nin H$, then $p \in H$ and $x \jj z < y \jj z$ is impossible. 
This shows that $\bgb_H$ satisfies (SP${}_\jj$). 
A similar, in fact easier, argument yields (SP${}_\mm$).

We leave Case $\SV{}$ and Case $\SH{}$ to the reader.
\end{proof}

Now the following statement is clear. 

\begin{lemma}\label{L:correspond}
The correspondence 
\[
   \gf \colon \bgb \to \Base(\bgb)
\]
is an order preserving bijection between
the order of I-congruences of $K$ and the order of 
down sets of $P^-$. We extend \gf by
$\zero \to \set{0}$ and  $\one \to P$.
Then \gf is an isomorphism between $\Con K$ and $\Down^- P$, 
the order of nonempty down sets of $P$.
\end{lemma}

\begin{lemma}\label{L:principal}
\gf and $\gf^{-1}$ both preserve the property of being principal.
\end{lemma}

\begin{proof}
Indeed, if the I-congruence \bgb of $K$ is principal, 
$\bgb = \con{x, y}$ for some $x < y \in K$,
then we must have $x, y \in S(p, q)$ for some 
$p < q \in P^-$ (otherwise, \bgb would not be an I-congruence). 
But in $S(p, q)$ (see Figure~\ref{F:Scongs}), the principal congruences are $\con{a_p, b_p}$ and $\con{a_q, b_q}$.
By Lemma~\ref{L:correspond}, 
we obtain that $\Base(\bgb) = {\Dg p}$ or $\Base(\bgb) = {\Dg q}$.

Conversely, if $\Base(\bgb) = {\Dg p}$, then $\bgb = \con{a_p, b_p}$.
\end{proof}

Now Theorem \ref{T:bounded} easily follows. Indeed, 
by Lemma~\ref{L:correspond}, \gf is an isomorphism 
between $\Con K$ and $\Down^- P$. Under this isomorphism,
by Lemma~\ref{L:principal}, principal congruences correspond
to principal down sets, so $\Princ K \iso P$, as claimed.

\subsection*{Note added in proof} 
R.\,W. Quackenbush has just sent me
a manuscript of his with P.\,P. P\'alfy, 
The representation of principal congruences,
accepted for publication in 1993. 
The final version was not submitted.

Anybody interested in the universal algebraic problems
of Section~\ref{S:Algebras} should read this manuscript.

\end{document}